 \documentclass[11pt]{amsart}
\usepackage{geometry}   
\usepackage{setspace}
\usepackage{graphicx}
\usepackage{amssymb, amsthm}
\usepackage{epstopdf}
 \newtheorem{theorem}{Theorem}[section]
    
   \newtheorem{lemma}[theorem]{Lemma}
    \newtheorem{proposition}[theorem]{Proposition}

    \theoremstyle{definition}

\def\lg{{\rm lg}} \def\Sal{{\rm Sal}} \def\BSal{\overline{\rm Sal}}
\def\int{{\rm int}}

\title{Convexity of parabolic subgroups in Artin groups}

\author{Ruth Charney and Luis Paris}

\thanks {R. Charney was partially supported by NSF grant DMS-1106726}

\begin{document}


\begin{abstract}
\noindent
We prove that any standard parabolic subgroup of any Artin group is convex with respect to the standard generating set.
\end{abstract}

\maketitle



\section{Introduction}

Let $S$ be a finite set. 
A \emph{Coxeter matrix} over $S$ is a square matrix $M=(m_{s,t})_{s,t \in S}$ indexed by the elements of $S$ satisfying  $m_{s,s}=1$ for all $s \in S$, and $m_{s,t} = m_{t,s} \in \{2,3,4, \dots\} \cup \{ \infty\}$ for all $s,t \in S$, $s \neq t$.
The \emph{Coxeter graph} which represents the above Coxeter matrix is the labelled graph $\Gamma = \Gamma(M)$ defined as follows.
(1) $S$ is the set of vertices of $\Gamma$.
(2) Two vertices $s,t \in S$ are joined by an edge if $m_{s,t} \ge 3$.
(3) This edge is labelled by $m_{s,t}$ if $m_{s,t} \ge 4$.

The \emph{Coxeter system} of $\Gamma$ is the pair $(W,S) = (W_\Gamma,S)$, where $S$ is the set of vertices of $\Gamma$, and $W$ is the group defined by the following presentation.
\[
W= \left\langle S\ \left|\ \begin{array}{c}
s^2=1 \text{ for all } s \in S\\
(st)^{m_{s,t}}=1 \text{ for all } s,t \in S,\ s \neq t \text{ and } m_{s,t} \neq \infty
\end{array}\right. \right\rangle\,.
\]
The group $W$ itself is called the \emph{Coxeter group} of $\Gamma$.

If $a,b$ are two letters and $m$ is an integer $\ge 2$, we set 
$$\Pi(a,b:m) = \begin{cases} (ab)^{\frac{m}{2}} &  \textrm{if $m$ is even} \cr
		a(ba)^{\frac{m-1}{2}} & \textrm{if $m$ is odd}
 \end{cases}$$
In other words, $\Pi(a,b:m)$ denotes the word $aba \cdots$ of length $m$.
Let $\Sigma= \{ \sigma_s \mid s \in S\}$ be an abstract set in one-to-one correspondence with $S$.
The \emph{Artin system} of $\Gamma$ is the pair $(A,\Sigma) = (A_\Gamma,\Sigma)$, where $A$ is the group defined by the following presentation.
\[
A = \langle \Sigma \mid \Pi(\sigma_s, \sigma_t: m_{s,t}) = \Pi(\sigma_t, \sigma_s: m_{s,t}) \text{ for all } s,t \in S,\ s \neq t \text{ and } m_{s,t} \neq \infty \rangle\,.
\]
The group $A$ itself is called the \emph{Artin group} of $\Gamma$.
Observe that the map $\Sigma \to S$, $\sigma_s \mapsto s$, induces an epimorphism $\theta: A_\Gamma \to W_\Gamma$.
The kernel of $\theta$ is called the \emph{colored Artin group} of $\Gamma$, and it is denoted by $CA_\Gamma$.

The Coxeter groups were introduced by Tits in his manuscript \cite{Tits1}.
The latter is one of the main sources for the celebrated Bourbaki book ``Groupes et alg\`ebres de Lie, Chapitres IV, V et VI'' \cite{Bourb1}.
They play an important role in many areas such as Lie theory, hyperbolic geometry, and, of course, group theory. Furthermore, there is a quite extensive literature on Coxeter groups. 
We recommend \cite{Davis1} for a detailed study of the subject.

The Artin groups were also introduced by Tits, as extensions of Coxeter groups \cite{Tits2}.
They are involved in several fields (singularities, knot theory, mapping class groups, and so on), and they have been the object of many papers in the last three decades.  Most results, however, involve only special classes of Artin groups, such as spherical type Artin groups (where the corresponding Coxeter group is finite) or right-angled Artin groups (where all $m_{s,t}=2,\infty$).  Artin groups as a whole are still poorly understood.
In particular, it is  not known if they are torsion free, or if they have solvable word problem (see \cite{GodPar1}).
The present paper is an exception to that rule since it concerns all Artin groups.

The flagship example of an Artin group is the braid group $\mathcal B_n$ on $n$ strands.
It is associated to the Coxeter graph $ \mathbb A_{n-1}$ depicted in Figure \ref{An}, and its associated Coxeter group is the symmetric group $\mathfrak S_n$.
\begin{figure}[tbh]\label{An}
\bigskip
\centerline{
\setlength{\unitlength}{0.5cm}
\begin{picture}(11,1.5)
\put(0,1){\includegraphics[width=5.5cm]{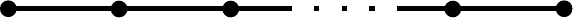}}
\put(0,0.2){\small $s_1$}
\put(2.1,0.2){\small $s_2$}
\put(4.2,0.2){\small $s_3$}
\put(8,0.3){\small $s_{n-2}$}
\put(10.2,0.3){\small $s_{n-1}$}
\end{picture}} 

\bigskip
\caption{The Coxeter graph $\mathbb A_{n-1}$.}
\end{figure}
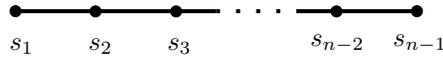

For $T \subset S$, we denote by $W_T$ the subgroup of $W$ generated by $T$, we denote by $\Gamma_T$ the full subgraph of $\Gamma$ spanned by $T$, we set $\Sigma_T = \{\sigma_s \mid s \in T\}$, and we denote by $A_T$ the subgroup of $A$ generated by $\Sigma_T$.
By \cite{Bourb1}, the pair $(W_T,T)$ is the Coxeter system of $\Gamma_T$, and by \cite{Lek1}, $(A_T,\Sigma_T)$ is the Artin system of $\Gamma_T$.
The group $A_T$ (resp. $W_T$) is called a \emph{standard parabolic subgroup} of $A$ (resp. of $W$).

Let $m$ and $n$ be two positive integers such that $m \le n$.
A fundamental example of a standard parabolic subgroup is the braid group $\mathcal B_m$ embedded in $\mathcal B_n$ via the homomorphism which sends the standard generators of $\mathcal B_m$ to the first $m-1$ standard generators of $\mathcal B_n$.

Let $G$ be a group, and let $S$ be a generating set for $G$.
An \emph{expression} for an element $\alpha \in G$ is a word $\widehat \alpha = s_1^{\varepsilon_1} \cdots s_\ell^{\varepsilon_\ell}$ on $S \sqcup S^{-1}$ which represents $\alpha$.
The \emph{length} of $\alpha$ (with respect to $S$), denoted by $\lg_S(\alpha)$, is the minimal length of an expression for $\alpha$.
A \emph{geodesic} for $\alpha$ is an expression of length $\lg_S(\alpha)$.  Geometrically, this corresponds to a geodesic path from $1$ to $\alpha$ in the Cayley graph of $(G,S)$.

Let $T$ be a subset of $S$, and let $H$ be the subgroup of $G$ generated by $T$.
We say that $H$ is \emph{convex} in $G$ with respect to $S$ if for all $\alpha \in H$ and all geodesic $\widehat \alpha = s_1^{\varepsilon_1} \cdots s_\ell^{\varepsilon_\ell}$ of $\alpha$, we have $s_1, \dots, s_\ell \in T$.
Or equivalently, the Cayley graph of $(H,T)$ is a convex subspace of the Cayley graph of $(G,S)$.  In particular, if $H$ is convex, then $\lg_S(\alpha) = \lg_T(\alpha)$ for all $\alpha \in H$, that is, $H$ is isometrically embedded in $G$.

The following is of importance in the study of Coxeter groups.

\begin{theorem}[Bourbaki \cite{Bourb1}]
Let $\Gamma$ be a Coxeter graph, let $(W,S)$ be its Coxeter system, and let $T$ be a subset of $S$.
Then $W_T$ is a convex subgroup of $W$ with respect to $S$.
\end{theorem}

In the present paper we prove that the same result holds for Artin groups.
That is:

\begin{theorem} \label{main}
Let $\Gamma$ be a Coxeter graph, let $(A,\Sigma)$ be its Artin system, and let $T$ be a subset of $S$.
Then $A_T$ is a convex subgroup of $A$ with respect to $\Sigma$.
\end{theorem}

Our original motivation for studying the convexity of parabolic subgroups in Artin groups comes from a question asked us by Arye Juh\'asz.
In a work in preparation \cite{Juhas1}, he provides a solution to the word problem for a certain class of Artin groups, and he proves that these groups are torsion free. 
His proof uses a condition on the groups that he calls the A-S condition, a form of convexity for a certain type of word. It follows immediately from Theorem~1.2 that this condition is satisfied by all Artin groups.
More generally, as for Coxeter groups, the study of Artin groups often goes through the study of their (standard) parabolic subgroups. So any result on these subgroups is likely to be useful in further developments in the theory of Artin groups.

Although Theorem \ref{main} may seem natural, it is a surprise for the experts.
As far as we know, it was not even known for the braid group $\mathcal B_m$ embedded in $\mathcal B_n$,
although the proof in this case is easy (see Proposition \ref{braids} below).  
Theorem  \ref{main}  also comes as a surprise because, in general, the family of standard generators (that is, $\Sigma$) is not the best for studying combinatorial questions on Artin groups.
In the most well understood case, when $A$ is of spherical type, a larger generating set is generally used.
This is called the set of \emph{simple elements} and we denote it by $\mathcal S$. 
It follows from \cite{Charn1} that  $(A_T, \mathcal S_T)$ is isometrically embedded in $(A,\mathcal S)$, but
the image is \emph{not} convex since there are $\mathcal S$-geodesics for elements of $A_T$ whose terms do not belong to $\mathcal S_T \sqcup \mathcal S_T^{-1}$.


\section{The proof}

\noindent
As promised, we start with the proof of Theorem  \ref{main}  in the particular case of the braid group $\mathcal B_m$ embedded in $\mathcal B_n$. 
We treat this case separately for two reasons. 
Firstly, because some readers may want to use Theorem  \ref{main}  in that case without necessarily learning all the background on Artin groups needed to prove our theorem.
Secondly, because the proof of Theorem  \ref{main}  is, in a sense, a (non trivial) extension of the proof of Proposition \ref{braids}.
So, the reader may want to keep in mind the proof of Proposition \ref{braids} when reading the proof of Theorem~\ref{main}.

Recall that a \emph{braid} on $n$ strands is an $n$-tuple $\beta=(b_1, \dots, b_n)$ of paths, $b_i: [0,1] \to \mathbb R^3$, such that 
\begin{itemize}
\item[(1)] 
$b_i(0)=(i,0,0)$ for all $i \in \{1, \dots, n\}$, and there is a permutation $w \in \mathfrak S_n$ such that $b_i(1)=(w(i),0,1)$ for all $i \in \{1, \dots, n\}$;
\item[(2)]
$(p_3 \circ b_i)(t) = t$ for all $i \in \{1, \dots, n\}$ and $t \in [0,1]$, where $p_3$ denotes the projection of $\mathbb R^3$ on the third coordinate; 
\item[(3)]
$b_i \cap b_j = \emptyset$ for all $i,j \in \{1, \dots, n\}$, $i \neq j$.
\end{itemize}
The isotopy classes of braids form a group, called the \emph{braid group} on $n$ strands and denoted by $\mathcal B_n$. 
By Artin \cite{Artin1,Artin2}, this group has the following presentation. 
\[
\mathcal B_n = \left\langle \sigma_1, \dots, \sigma_{n-1} \ \left| \begin{array}{c}
\sigma_i \sigma_{i+1} \sigma_i = \sigma_{i+1} \sigma_i \sigma_{i+1} \text{ for } 1 \le i \le n-2\\
\sigma_i \sigma_j = \sigma_j \sigma_i \text{ for } |i-j| \ge 2
\end{array} \right. \right\rangle \,.
\]
In other words, $\mathcal B_n$ is the Artin group associated to the Coxeter graph $\mathbb A_{n-1}$ depicted in Figure \ref{An}.

\begin{proposition} \label{braids}
Let $m,n \in \mathbb N$, $m \le n$.
Then $\mathcal B_m$ is a convex subgroup of $\mathcal B_n$ with respect to $\{\sigma_1, \dots, \sigma_{n-1}\}$.
\end{proposition}

\begin{proof} 
Let $p_{1,3}: \mathbb R^3 \to \mathbb R^2$ denote the projection on the first and third coordinates.
Let $\beta = (b_1, \dots, b_n)$ be a braid.
We project each $b_i$ on the plane $\mathbb R^2$ via $p_{1,3}$ and, up to isotopy, we may suppose that these projections form only finitely many regular double crossings.
As usual, we indicate in each crossing which strand goes over the other. 
We obtain in this way a \emph{braid diagram} which represents the isotopy class of $\beta$. 
Recall also that each generator $\sigma_i$ has a ``canonical'' diagram with precisely one crossing.

Let $\widehat \alpha = \sigma_{i_1}^{\varepsilon_1} \cdots \sigma_{i_\ell}^{\varepsilon_\ell}$ be a word which represents an element $\alpha \in \mathcal B_n$.
By concatenating the canonical diagrams of the $\sigma_i^{\varepsilon_i}$'s, we obtain a diagram of $\alpha$ with precisely $\ell$ crossings.
Conversely, by applying standard methods, from a diagram of $\alpha$ with $\ell$ crossings, we can define a (non unique) word $\widehat \alpha = \sigma_{i_1}^{\varepsilon_1} \cdots \sigma_{i_\ell}^{\varepsilon_\ell}$ of length $\ell$ which represents $\alpha$.

Let $\alpha \in \mathcal B_m$, and let $\widehat \alpha = \sigma_{i_1}^{\varepsilon_1} \cdots \sigma_{i_\ell}^{\varepsilon_\ell}$ be a geodesic of $\alpha$ with respect to the generating set of $\mathcal B_n$. 
Let $D$ be the diagram of $\alpha$ obtained from $\widehat \alpha$.
By removing the strands $b_{m+1}, \dots, b_n$ from $D$, we obtain another diagram $D'$ of $\alpha$, but now viewed as an element of $\mathcal B_m$.  Since no new crossings are introduced by this procedure, $D'$ has at most  $\ell$ crossings.  But $\widehat \alpha$ was geodesic, so the number of crossings cannot be less than $\ell$.  
This means that the only strands involved in the crossings in $D$ are $b_1, \dots, b_m$, and therefore $\sigma_{i_1}, \dots, \sigma_{i_\ell} \in \{\sigma_1, \dots, \sigma_{m-1}\}$.
\end{proof}

We turn now to the proof of Theorem \ref{main}.
We fix a Coxeter graph $\Gamma$, and denote by $(W,S)$ its Coxeter system, and by $(A, \Sigma)$ its Artin system.

Let $X,Y$ be two subsets of $S$, and let $w$ be an element of $W$.
We say that $w$ is \emph{$(X,Y)$-minimal} if it is of minimal length among the elements of the double-coset $W_X w W_Y$.
The following will be implicitly used throughout the paper.

\begin{proposition} [Bourbaki \cite{Bourb1}] \label{minimal}  Let $X,Y$ be two subsets of $S$, and let $w \in W$. 
\begin{enumerate}
\item There exists a unique $(X,Y)$-minimal element lying in $W_X w W_Y$.
\item  The following are equivalent.
\begin{itemize}
\item $w$ is $(\emptyset,Y)$-minimal,
\item  $\lg_S (ws) > \lg_S (w)$ for all $s \in Y$,
\item  $\lg_S (wu) = \lg_S(w) + \lg_S(u)$ for all $u \in W_Y$.
\end{itemize}
\item The following are equivalent.
\begin{itemize}
\item  $w$ is $(X,\emptyset)$-minimal,
\item $\lg_S (sw) > \lg_S (w)$ for all $s \in X$, 
\item $\lg_S (uw) = \lg_S(u) + \lg_S(w)$ for all $u \in W_X$.
\end{itemize}
\end{enumerate}
\end{proposition}

Now, set  $\mathcal S^f = \{X \subset S \mid W_X \text{ is finite}\}$.
Then the following can be easily proved using the previous proposition (see for instance \cite[Lemma 3.2]{Paris1}).

\begin{lemma} \label{poset}
Let $\preceq$ be the relation on $W \times \mathcal S^f$ defined as follows.
Set $(u,X) \preceq (v,Y)$ if the following three conditions hold,
\begin{enumerate}
\item $X \subset Y$, \item $v^{-1}u \in W_Y$, and \item $v^{-1}u$ is $(\emptyset, X)$-minimal.
\end{enumerate}
Then $\preceq$ is a partial order relation on $W \times \mathcal S^f$.
\end{lemma}

Recall that the \emph{derived complex} of a partially ordered set $(E, \le)$ is defined to be the abstract simplicial complex made of the finite nonempty chains of $E$.
The \emph{Salvetti complex} of $\Gamma$, denoted by $\Sal (\Gamma)$, is defined to be the geometric realization of the derived complex of $(W \times \mathcal S^f, \preceq)$.
Note that the action of $W$ on $W \times \mathcal S^f$ defined by $w \cdot (u,X) = (wu,X)$ leaves invariant the order relation $\preceq$, hence it induces a ``natural'' (free and properly discontinuous) action of $W$ on $\Sal(\Gamma)$.
The quotient space will be denoted by $\BSal(\Gamma) = \Sal(\Gamma)/W$.

The spaces $\Sal(\Gamma)$ and $\BSal(\Gamma)$ admit cellular decompositions. Recall that a finite Coxeter group $W_X$ can be realized as an orthogonal reflection group acting on $\mathbb R^k, k={|X|}$.
The \emph{Coxeter cell} for $W_X$ is the convex hull of the orbit of a generic point $p$ in $\mathbb R^k$.  Viewed as a combinatorial object, it is independent of choice of the point $p$.  Its 1-skeleton is the Cayley graph of $(W_X,X)$.
 The cell decompositions of 
$\Sal(\Gamma)$ and $\BSal(\Gamma)$, which we now describe, are made up of Coxeter cells for the subgroups $W_X, X \in \mathcal S^f$.

For $(u,X) \in W \times \mathcal S^f$, we set
\[
C(u,X) = \{(v,Y) \in W \times \mathcal S^f \mid (v,Y) \preceq (u,X)\}\,,
\]
and we denote by $\mathbb B(u,X)$ the simplicial subcomplex of $\Sal(\Gamma)$ spanned by $C(u,X)$.
It is shown in \cite{Paris1} that $\mathbb B(u,X)$ is a closed disc of dimension $|X|$ for all $(u,X) \in W \times \mathcal S^f$;
in fact it is naturally isomorphic to the Coxeter cell for $W_X$.  The set $\{\mathbb B(u,X) \mid (u,X) \in W \times \mathcal S^f \}$ is a regular cellular decomposition of $\Sal(\Gamma)$ (see also \cite{Salve1, ChaDav1}).

Observe that every $w \in W \setminus \{1\}$ sends the closed disc $\mathbb B(u,X)$  homeomorphically onto $\mathbb B(wu,X)$, and that $\int (\mathbb B(u,X)) \cap \int(\mathbb B(wu,X)) = \emptyset$, for all $(u,X) \in W \times \mathcal S^f$.
It follows that the cellular decomposition of $\Sal(\Gamma)$ induces a cellular decomposition of the quotient $\BSal(\Gamma) = \Sal(\Gamma)/W$.  This decomposition has one cell for each 
$X \in \mathcal S^f$, corresponding to the orbit of $\mathbb B(1,X)$ under $W$. 
The closure of this cell will be denoted by $\bar{\mathbb B}(X)$.
Note that this cellular decomposition is not regular in general.

The $k$-skeletons of $\Sal(\Gamma)$ and $\BSal(\Gamma)$ for $k=0,1,2$ can be described as follows.

{\bf The 0-skeleton.}
Let $u \in W$.
Then $C(u,\emptyset) = \{(u,\emptyset)\}$, and $\mathbb B(u,\emptyset)$ is a vertex of $\Sal(\Gamma)$.
We denote it by $x(u)$.
So, the $0$-skeleton of $\Sal(\Gamma)$ is the set $\{x(u) \mid u \in W\}$, which is in one-to-one correspondence with $W$.
The $0$-skeleton of $\BSal(\Gamma)$ is reduced to a single point, $\bar{\mathbb B}(\emptyset)$, that we denote by $x_0$.

{\bf  The 1-skeleton.}
Let $u \in W$ and $s \in S$.
Then $C(u,\{s\}) = \{(u,\emptyset), (us, \emptyset), (u,\{s\})\}$, and $\mathbb B (u,\{s\})$ is an edge of $\Sal(\Gamma)$
connecting $x(u)$ to $x(us)$.
We denote it by $a(u,s)$ and orient it from $x(u)$ to $x(us)$.
Note that there is another edge connecting $x(u)$ to $x(us)$, namely $a(us,s)$, but this edge is oriented in the other direction (see Figure \ref{1-skeleton}).  Observe that the action of $W$ preserves orientations of edges.
Since all edges are of this form, we see that the 1-skeleton of $\Sal(\Gamma)$ is just the Cayley graph of $(W,S)$.  
 Descending to $\BSal(\Gamma)$, the 1-skeleton consists of a  loop at $x_0$ formed by the edge $\bar a_s =\bar{\mathbb B}(\{s\})$, for each $s \in S$ (see Figure \ref{1-skeleton}).

\begin{figure}[tbh]\label{1-skeleton}
\bigskip
\centerline{
\setlength{\unitlength}{0.5cm}
\begin{picture}(16,4)
\put(1,1){\includegraphics[width=7cm]{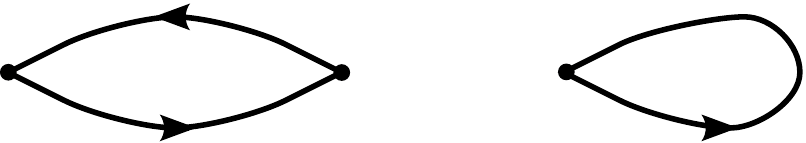}}
\put(0,1.2){\small $x(u)$}
\put(3,0.2){\small $a(u,s)$}
\put(3,3.6){\small $a(v,s)$}
\put(7,2.5){\small $x(v)$}
\put(7.5,1.8){\small $\shortparallel$}
\put(6.8,1){\small $x(us)$}
\put(10.3,1.4){\small $x_0$}
\put(14.9,1){\small $\bar a_s$}
\end{picture}} 

\bigskip
\caption{$1$-skeletons of $\Sal(\Gamma)$ and $\BSal(\Gamma)$.}
\end{figure}

{\bf The 2-skeleton.}
Let $s,t \in S$, $s \neq t$.
Set $X=\{s,t\}$.
First, notice that $X \in \mathcal S^f$ if and only if $m_{s,t} \neq \infty$.
Now, assume $m_{s,t} \neq \infty$, set $m=m_{s,t}$, and take $u \in W$.
Then $\mathbb B(u,X)$ is isomorphic to the Coxeter cell for $W_X$.  Namely, it is a $2m$-gon
with vertices $\{x(uw) \mid w \in W_X\}$.  
The boundary of $\mathbb B(u,X)$ is the loop
\[
a(u,s) a(us,t) \cdots a(u \Pi(s,t:m-1),r) a(u\Pi(t,s:m-1),r')^{-1} \cdots a(ut,s)^{-1} a(u,t)^{-1}
\]
where $r=t$ and $r'=s$ if $m$ is even, and $r=s$ and $r'=t$ if $m$ is odd 
(see Figure 3).  
Hence, $\bar{\mathbb B}(X)$ is a $2$-cell whose boundary is
\[
\bar a_s \bar a_t \cdots \bar a_r \bar a_{r'}^{-1} \cdots \bar a_s^{-1} \bar a_t^{-1} = \Pi(\bar a_s,\bar a_t:m)\Pi(\bar a_t,\bar a_s:m)^{-1}\,.
\]

\begin{figure}[tbh]
\label{2-skeleton}
\bigskip
\centerline{
\setlength{\unitlength}{0.5cm}
\begin{picture}(22,10)
\put(2,1.5){\includegraphics[width=10cm]{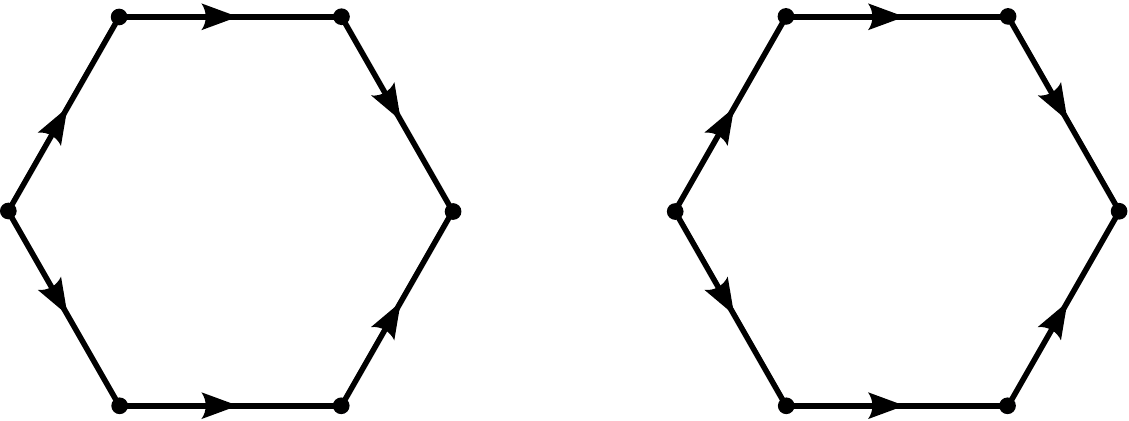}}
\put(0.7,3){\small $a(u,t)$}
\put(0.7,7){\small $a(u,s)$}
\put(4.8,0.7){\small $a(ut,s)$}
\put(4.8,9.1){\small $a(us,t)$}
\put(9.5,3){\small $a(uts,t)$}
\put(9.5,7){\small $a(ust,s)$}
\put(14,3){\small $\bar a_t$}
\put(14,7){\small $\bar a_s$}
\put(17.5,0.7){\small $\bar a_s$}
\put(17.5,9.1){\small $\bar a_t$}
\put(21.4,3){\small $\bar a_t$}
\put(21.4,7){\small $\bar a_s$}
\end{picture}} 

\bigskip
\caption{$2$-skeletons of $\Sal(\Gamma)$ and $\BSal(\Gamma)$.}
\end{figure}

From the above it follows that the $2$-skeleton of $\BSal(\Gamma)$ is precisely the Cayley $2$-complex associated with the standard presentation of $A$.

\begin{proposition} 
We have $\pi_1(\BSal(\Gamma),x_0) = A$, $\pi_1(\Sal(\Gamma),x(1)) = CA$, and the exact sequence associated to the regular covering $\Sal(\Gamma) \to \BSal(\Gamma)$ is 
\[
1 \longrightarrow CA \longrightarrow A \longrightarrow W \longrightarrow 1.
\]
\end{proposition}

Now, assume we are given a subset $T$ of $S$, and set $\mathcal S_T^f = \{ X \in \mathcal S^f \mid X \subset T\}$.
Observe that the embedding $W_T \times \mathcal S_T^f \hookrightarrow W \times \mathcal S^f$ induces an embedding $\iota_T : \Sal(\Gamma_T) \hookrightarrow \Sal(\Gamma)$.
The following provides our main tool for proving Theorem \ref{main}.

\begin{theorem} [Godelle, Paris \cite{GodPar2}] \label{GodPar}
The embedding $\iota_T : \Sal(\Gamma_T) \hookrightarrow \Sal(\Gamma)$ admits a retraction $\pi_T: \Sal(\Gamma) \to \Sal(\Gamma_T)$.
\end{theorem}

In the proof of Theorem \ref{main}, we will need the following explicit description of the map $\pi_T$.
Let $(u,X) \in W \times \mathcal S^f$.
Write $u=u_0u_1$, where $u_0 \in W_T$ and $u_1$ is $(T, \emptyset)$-minimal.
Set $X_0 = T \cap u_1 X u_1^{-1}$.
Note that, since $u_1 W_X u_1^{-1}$ is finite, $W_{X_0}$ is finite, hence $X_0 \in \mathcal S_T^f$.
Then we set $\pi_T(u,X) = (u_0,X_0)$.
It is proved in \cite{GodPar2} that the  map $\pi_T: W \times \mathcal S^f \to W_T \times \mathcal S_T^f$ induces a continuous map $\pi_T: \Sal(\Gamma) \to \Sal(\Gamma_T)$ which is a retraction of $\iota_T$.

Now, the following lemma follows from the above description.

\begin{lemma} \label{projection} 
 Let $u \in W$.
Write $u=u_0u_1$, where $u_0 \in W_T$ and $u_1$ is $(T, \emptyset)$-minimal.  Then
\begin{enumerate}
\item   $\pi_T(x(u)) = x(u_0)$.
\item Let $s \in S$ and set  $t = u_1 s u_1^{-1}$.  If $t \in T$, 
then $\pi_T(a(u,s)) = a(u_0,t)$. If $t \notin T$, then  $\pi_T(a(u,s)) = x(u_0)$.
\end{enumerate}
\end{lemma}

\begin{proof} 
We leave Part (1) to the reader and turn to Part (2).
We take $u \in W$ and $s \in S$, and define $u_0,u_1$ and $t$ as in the Lemma.
Recall that $C(u,\{s\}) = \{(u, \emptyset), (us, \emptyset), (u,\{s\})\}$, and that $a(u,s)$ is the edge of $\Sal(\Gamma)$ spanned by $C(u,\{s\})$.

Assume first that $t \in T$.
We have $us = u_0 u_1 s = u_0 t u_1$, $u_0t \in W_T$ and $u_1$ is $(T,\emptyset)$-minimal, hence $\pi_T((u,\emptyset)) = (u_0,\emptyset)$, $\pi_T((us,\emptyset)) = (u_0t,\emptyset)$, and $\pi_T((u,\{s\})) = (u_0,\{t\})$.  Therefore, $\pi_T(a(u,s)) = a(u_0,t)$.

Assume now that $t \not \in T$.  We claim that $u_1s$ is $(T,\emptyset)$-minimal.
If $\lg_S(u_1s)< \lg_S(u_1)$, then this is clear since $u_1$ is $(T,\emptyset)$-minimal.
Suppose that $\lg_S(u_1s) > \lg_S(u_1)$.
If $u_1s$ were not $(T,\emptyset)$-minimal, then there would exist $s_0 \in T$ such that $\lg_S(s_0u_1s) < \lg_S(u_1s)$ and since $u_1$ is $(T,\emptyset)$-minimal,   $\lg_S(s_0u_1) > \lg_S(u_1)$.   By   the ``folding condition'' (see  \cite[Chap. II, Sec.~3]{Brown1}), this would imply that  $s_0  u_1 s = u_1$, that is, $t=u_1 s u_1^{-1} = s_0 \in T$, a contradiction.

Thus, if $t \not \in T$, then $us=u_0u_1s$, $u_0 \in W_T$, and $u_1s$ is $(T,\emptyset)$-minimal.
It follows that $\pi_T((u,\emptyset)) = \pi_T((us,\emptyset)) = \pi_T ((u,\{s\})) = (u_0, \emptyset)$, hence $\pi_T(a(u,s))= x(u_0)$.
\end{proof}

\begin{proof}[Proof of Theorem \ref{main}]
Let $\widehat \alpha = \sigma_{s_1}^{\varepsilon_1} \cdots \sigma_{s_\ell}^{\varepsilon_\ell} \in (\Sigma \sqcup \Sigma^{-1})^*$ be a word in the alphabet $\Sigma \sqcup \Sigma^{-1}$.
Set
\[
\bar \gamma (\widehat \alpha) = \bar a_{s_1}^{\varepsilon_1} \cdots \bar a_{s_\ell}^{\varepsilon_\ell}\,.
\]
This is a loop in $\BSal(\Gamma)$ based at $x_0$.
Note that, if $\alpha$ is the element of $A$ represented by the word $\widehat \alpha$, then $\alpha$ is the element of $A = \pi_1(\BSal(\Gamma), x_0)$ represented by the loop $\bar \gamma(\widehat \alpha)$.

Denote by $\gamma(\widehat \alpha)$ the lift of $\bar \gamma (\widehat \alpha)$ in $\Sal(\Gamma)$ with initial point $x(1)$.
The path $\gamma(\widehat \alpha)$ can be described as follows.
For $i \in \{0,1, \dots, \ell\}$ we set $u_i = s_1^{\varepsilon_1} s_2^{\varepsilon_2} \cdots s_i^{\varepsilon_i}=s_1s_2 \cdots s_i  \in W$.
For $i \in \{1, \dots, \ell\}$, we set $a_i=a(u_{i-1},s_i)$ if $\varepsilon_i = 1$ and $a_i=a(u_i,s_i)$ if $\varepsilon_i = -1$.
Then
\[
\gamma(\widehat \alpha) = a_1^{\varepsilon_1} \cdots a_\ell^{\varepsilon_\ell}\,.
\]

Let $i \in \{1, \dots, \ell\}$.
Write $u_i = v_iw_i$, where $v_i \in W_T$ and $w_i$ is $(T,\emptyset)$-minimal.  If $\varepsilon_i =1$, let $t_i = w_{i-1}s_iw_{i-1}^{-1}$.   If $\varepsilon_i =-1$,  let  $t_i = w_is_iw_i^{-1}$.  Define
\[
\tau_i = \begin{cases}  \sigma_{t_i}^{\varepsilon_i} &\textrm{if $t_i \in T$}\cr
	1 & \textrm{if $t_i \notin T$}
	\end{cases}
\]
$$\widehat \tau = \tau_1 \tau_2 \cdots \tau_\ell \in (\Sigma_T \sqcup \Sigma_T^{-1})^*.$$
Note that $\lg(\widehat \tau) \le \lg (\widehat \alpha)$.

By Lemma \ref{projection}, the projection of $\pi_T(a_i)$
in $\BSal(\Gamma_T)$ is the edge $\bar a_{t_i}$ if $t_i \in T$, and the vertex $x_0$ otherwise.
Thus the image of $\pi_T(\gamma(\widehat \alpha))$ in $\BSal(\Gamma_T)$ is the loop
corresponding to $\widehat \tau$.  In other words,  $\pi_T(\gamma(\widehat \alpha)) = \gamma (\widehat \tau )$.

{\bf Claim 1.}
{\it If $\widehat \alpha \in (\Sigma_T \sqcup \Sigma_T^{-1})^*$, then $\widehat \tau = \widehat \alpha$.}

{\it Proof of Claim 1.}
In this case we have $u_i \in W_T$, hence $v_i=u_i$ and $w_i=1$ for all $i \in \{0,1, \dots, \ell\}$.
It follows that $t_i=s_i \in T$, so $\tau_i = \sigma_{s_i}^{\varepsilon_i}$ for all $i$.
Hence,
\[
\widehat \tau = \tau_1 \tau_2 \cdots \tau_\ell = \sigma_{s_1}^{\varepsilon_1} \sigma_{s_2}^{\varepsilon_2} \cdots \sigma_{s_\ell}^{\varepsilon_\ell} = \widehat \alpha \,.
\]

{\bf Claim 2.}
{\it Suppose $\widehat \alpha \in (\Sigma \sqcup \Sigma^{-1})^*$ represents the element $\alpha \in A_T$. 
Then $\widehat \tau \in (\Sigma_T \sqcup \Sigma_T^{-1})^*$ also represents $\alpha$.}

{\it Proof of Claim 2.}
Choose a word $\widehat \alpha' \in (\Sigma_T \sqcup \Sigma_T^{-1})^*$ which represents $\alpha$.
By construction, $\bar \gamma (\widehat \alpha')$ and $\bar \gamma(\widehat \alpha)$ represent $\alpha$, hence they are homotopic in $\BSal(\Gamma)$ relative to $x_0$. 
So, $\gamma (\widehat \alpha')$ is homotopic to $\gamma(\widehat \alpha)$ relative to endpoints.
It follows that $\pi_T(\gamma (\widehat \alpha')) = \gamma (\widehat \alpha')$ is homotopic to $\pi_T (\gamma (\widehat \alpha)) = \gamma (\widehat\tau)$ in $\Sal(\Gamma_T)$.  
We conclude that $\widehat \tau$ and $\widehat \alpha'$ represent the same element of $A_T$, namely, $\alpha$.

{\bf Claim 3.}
{\it Let $\widehat \alpha \in (\Sigma \sqcup \Sigma^{-1})^*$.
If $\lg ( \widehat \tau ) = \lg (\widehat \alpha)$, then $\widehat \alpha \in (\Sigma_T \sqcup \Sigma_T^{-1})^*$.}

{\it Proof of Claim 3.}
In order to have $\lg(\widehat \tau ) = \lg (\widehat \alpha)$, we must have $t_i \in T$ for all $i \in \{1, \dots, \ell\}$.  We will  prove by induction on $i$ that $s_i$ also lies in  $T$ for all $i $.

Let $i=1$.
Suppose first that $\varepsilon_1=1$.  Then $u_0=w_0=1$, so $s_1=w_0s_1w_0^{-1}=t_1 \in T$.
Suppose now that $\varepsilon_1=-1$.
If we had $s_1 \not \in T$, then $s_1$ would be $(T,\emptyset)$-minimal and we would have $u_1=w_1=s_1$.  
In this case, $s_1= w_1 s_1 w_1^{-1} = t_1 \in T$, a contradiction.
So, $s_1 \in T$.

Now assume that $i \ge 2$.  By induction,
we have $u_{i-1} = s_1 \cdots s_{i-1} \in W_T$.
Suppose first that $\varepsilon_i=1$.
Since $u_{i-1} \in W_T$, we have $v_{i-1} = u_{i-1}$ and $w_{i-1} = 1$, hence $s_i = t_i \in T$.
Suppose now that $\varepsilon_i=-1$.
If we had $s_i \not\in T$, then we would have $v_i =u_{i-1}$ and $w_i=s_i$, which would imply that $s_i = t_i \in T$, a contradiction.
So, $s_i \in T$.

We can now complete  the proof of the theorem.
Let $\alpha \in A_T$, and let $\widehat \alpha \in (\Sigma \sqcup \Sigma^{-1})^*$ be a geodesic for $\alpha$.
By Claim~2, $\widehat \tau$ is also an expression for $\alpha$.
By construction, $\lg (\widehat \tau) \le \lg (\widehat \alpha)$, hence $\lg (\widehat \tau) = \lg (\widehat \alpha)$ since $\widehat \alpha$ is a geodesic.
By Claim 3 we conclude that $\widehat \alpha \in (\Sigma_T \sqcup \Sigma_T^{-1})^*$.
\end{proof}


\bigskip
\bigskip\noindent {\bf Ruth Charney,}

\smallskip\noindent
Department of Mathematics, MS 050, Brandeis University, 415 South Street, Waltham, MA 02453, USA.

\smallskip\noindent
E-mail: {\tt charney@brandeis.edu}

\bigskip\noindent {\bf Luis Paris,}

\smallskip\noindent 
Universit\'e de Bourgogne, Institut de Math\'ematiques de Bourgogne, UMR 5584 du CNRS, B.P. 47870, 21078 Dijon cedex, France.

\smallskip\noindent
E-mail: {\tt lparis@u-bourgogne.fr}



\begin{thebibliography}{99}

\small

\bibitem{Artin1}
{\bf E. Artin.} 
{\it Theorie der Z\"opfe.}
Abhandlungen Hamburg  {\bf 4} (1925), 47-72.

\bibitem{Artin2}
{\bf E. Artin.}
{\it Theory of braids.}
Ann. of Math. (2)  {\bf 48} (1947), 101--126.

\bibitem{Bourb1}
{\bf N. Bourbaki.}
{\it El\'ements de math\'ematique. Fasc. XXXIV. Groupes et alg\`ebres de Lie. Chapitre IV: Groupes de Coxeter et syst\`emes de Tits. Chapitre V: Groupes engendr\'es par des r\'eflexions. Chapitre VI: Syst\`emes de racines.}
Actualit\'es Scientifiques et Industrielles, No. 1337, Hermann, Paris, 1968. 

\bibitem{Brown1}
{\bf K.\,S. Brown.}
{\it Buildings.}
Springer-Verlag, New York, 1989.

\bibitem{Charn1}
{\bf R. Charney.}
{\it Geodesic automation and growth functions for Artin groups of finite type.}
Math. Ann. {\bf 301} (1995), no. 2, 307--324.

\bibitem{ChaDav1}
{\bf R. Charney, M.\,W. Davis.}
{\it Finite $K(\pi,1)$'s for Artin groups.}
Prospects in topology (Princeton, NJ, 1994), 110--124, Ann. of Math. Stud., 138, Princeton Univ. Press, Princeton, NJ, 1995. 

\bibitem{Davis1}
{\bf M.\,W. Davis.}
{\it The geometry and topology of Coxeter groups.}
London Mathematical Society Monographs Series, 32. Princeton University Press, Princeton, NJ, 2008.

\bibitem{GodPar2}
{\bf E. Godelle, L. Paris.}
{\it $K(\pi,1)$ and word problems for infinite type Artin-Tits groups, and applications to virtual braid groups.}
Math. Z. {\bf 272} (2012), no. 3-4, 1339--1364.

\bibitem{GodPar1}
{\bf E. Godelle, L. Paris.}
{\it Basic questions on Artin-Tits groups.}
Configuration Spaces. Geometry, Combinatorics and Topology, pp. 299--311. Edizioni della Normale, Scuola Normale Superiore Pisa, 2012.

\bibitem{Juhas1}
{\bf A. Juh\'asz.}
{\it On relatively extra large Artin groups and their relative asphericity.}
In preparation.

\bibitem{Lek1}
{\bf H. van der Lek.}
{\it The homotopy type of complex hyperplane complements.}
Ph. D. Thesis, Nijmegen, 1983.

\bibitem{Paris1}
{\bf L. Paris.}
{\it $K(\pi,1)$ conjecture for Artin groups.}
Ann. Fac. Sci. Toulouse Math., to appear.

\bibitem{Salve1}
{\bf M. Salvetti.}
{\it The homotopy type of Artin groups.}
Math. Res. Lett. {\bf 1} (1994), no. 5, 565--577.

\bibitem{Tits1}
{\bf J. Tits.}
{\it Groupes et g\'eom\'etries de Coxeter.}
Institut des Hautes Etudes Scientifiques, Paris, 1961.

\bibitem{Tits2}
{\bf J. Tits.}
{\it Normalisateurs de tores. I. Groupes de Coxeter étendus.}
J. Algebra {\bf 4} (1966), 96--116.

\end{thebibliography}
\end{document}